\documentclass{amsart}
\usepackage{amsmath, amsthm, amssymb, amscd, amsfonts,enumerate}
\usepackage[utf8]{inputenc}
\usepackage[mathscr]{eucal}
\usepackage{color}
\usepackage{float}
\usepackage{cite}
\usepackage[breaklinks]{hyperref}
\newtheorem{theorem}{Theorem}[section]
\newtheorem{lemma}[theorem]{Lemma}

\newtheorem{definition}[theorem]{Definition}
\usepackage{longtable}
\usepackage{threeparttablex}

\theoremstyle{remark}

\newtheorem{exm}{Example}

\title{On Right $S$-Noetherian Rings and $S$-Noetherian Modules}

\makeatletter
\renewcommand*{\@fnsymbol}[1]{\ifcase#1\or1\else\@arabic{\numexpr#1\relax}\fi}

\newcommand{\F}{\mathcal{F}}

\begin{document}

	\author{Zehra Bİlgİn}
	\address{Department of Mathematics, Yıldız Technical University, 34210, Esenler,
		İstanbul, TURKEY}
	\email{zbilgin@yildiz.edu.tr}
	
		\author{Manuel L. Reyes}
		\address{Department of Mathematics, Bowdoin College, 
			Brunswick, Me 04011-8486, USA}
		\email{reyes@bowdoin.edu}
		\thanks{M.L.~Reyes was partially supported by NSF grant no.\ DMS-1407152.}

\author{Ünsal Tekİr}
\address{Department of Mathematics, Marmara University, 34722, Göztepe,  
	İstanbul, TURKEY}
\email{utekir@marmara.edu.tr}

	\date{August 11, 2017}


	\begin{abstract}
		In this paper we study right $S$-Noetherian rings and modules, extending notions introduced by Anderson and Dumitrescu in commutative algebra to noncommutative rings. Two characterizations of right $S$-Noetherian rings are given in terms of completely prime right ideals and point annihilator sets. We also prove an existence result for completely prime point annihilators of certain $S$-Noetherian modules with the following consequence in commutative algebra: If a module $M$ over a commutative ring is $S$-Noetherian with respect to a multiplicative set $S$ that contains no zero-divisors for $M$, then $M$ has an associated prime.
	\end{abstract}

	\maketitle
	
	\section{Introduction}
	
	Throughout this paper, all rings are associative with identity and all modules are unitary. In algebra, Noetherian rings have a distinguished place. Cohen's celebrated theorem gives a characterization of commutative Noetherian rings through prime ideals: A commutative ring is  Noetherian if and only if each of its prime ideals is finitely generated (see \cite[Theorem 2]{C}). 
	The problem of finding suitable extensions of commutative results in noncommutative algebra is often not as straightforward as it seems. For instance, if one uses the typical definition of a \emph{prime ideal} in a noncommutative ring (see \cite[Definition 5.1]{L}), then the naive extension Cohen's Theorem does not hold, as shown in \cite[Remark~2.11]{R3}. Nevertheless, extensions of Cohen's Theorem to the context of right ideals in noncommutative rings were achieved using using completely prime right ideals and Oka families of right ideals in \cite{R1}, and point annihilator sets in \cite{R2}. Following \cite{R1}, we say that a proper right ideal $P$ of $R$ is a \emph{completely prime right ideal} if for any $a,b\in R$ satisfying $aP\subseteq P$ and $ab\in P$ we have $a\in P$ or $b\in P$. Also, we recall that an \emph{Oka family of right ideals} (or a \emph{right Oka family}) in a (noncommutative) ring $R$ is a family $\mathcal{F}$ of right ideals with $R\in\mathcal{F}$ such that for any right ideal $I$ of $R$ and any element $a$ of $R$, 
	$$I+aR, a^{-1}I\in \mathcal{F} \Rightarrow I\in \mathcal{F},$$ 
where $a^{-1}I = \{r \in R : ar \in I\}$.
It was shown in~\cite[Theorem~3.8]{R1} that a ring is right noetherian if and only if each of its completely prime right ideals is finitely generated, with the proof relying crucially on the fact that the set of finitely generated right ideals in any ring $R$ is an Oka family.

	A generalization of commutative Noetherian rings, known as $S$-Noetherian rings, was defined by D.\,D.~Anderson and T.~Dumitrescu in \cite{A}. Several authors have continued the study of such rings, as in \cite{AS1}, \cite{AS2}, \cite{LO}, \cite{Z}. Let $R$ be a commutative ring with a multiplicative subset $S$. An ideal $I$ of $R$ is called \emph{$S$-finite} if $Is\subseteq J\subseteq I$ for some $s\in S$ and some finitely generated ideal $J$. The ring $R$ is called \emph{$S$-Noetherian} if each ideal of $R$ is $S$-finite. In \cite{A}, among other things, Anderson and Dumitrescu obtained an $S$-version of Cohen's Theorem. They proved that a commutative ring $R$ is $S$-Noetherian if and only if every prime ideal of $R$ disjoint from $S$ is $S$-finite (see \cite[Corollary 5]{A}). 
	
	In this paper, we seek to extend the concept of $S$-finiteness from~\cite{A} to a noncommutative setting, with the goal of retaining an $S$-version of Cohen's Theorem. As in \cite[Theorem 6.2]{R1} and \cite[Theorem 4.5]{R2}, we use Oka families as our tool to achieve this goal. The noncommutative extensions of the various $S$-finiteness properties that are suitable to this goal are as follows.
	Recall that a subset $S$ of a (not necessarily commutative) ring $R$ is said to be a \emph{multiplicative subset} if $1\in S$ and $ab\in S$ for any $a,b\in S$. 
	      \begin{definition}
	      	Let $R$ be a ring and $S\subseteq R$ a multiplicative subset. A right $R$-module $M$ is is said to be \emph{$S$-finite} if $Ms\subseteq F$ for some $s\in S$ and some finitely generated submodule $F$ of $M$, and $M$ is said to be \emph{$S$-Noetherian} if every submodule of $M$ is an $S$-finite module.
		A right ideal $I$ of $R$ is \emph{$S$-finite} if it is $S$-finite as a right $R$-module, and $R$ is called \emph{right $S$-Noetherian} if it the right $R$-module $M = R$ is $S$-noetherian.
	      \end{definition} 
	
	Our major results are as follows. In Theorem \ref{E:Cohen}, we give a noncommutative extension of Anderson and Dimitrescu's $S$-version of Cohen's theorem, characterizing right S-Noetherian rings in terms of $S$-finiteness of completely prime right ideals.
	We provide some equivalent conditions for a right $R$-module to be $S$-Noetherian in Theorem \ref{E:Noether} similar to the characterization of commutative Noetherian rings with the help of ascending chain condition and maximal condition.  
	  We apply point annihilator set for modules in the sense of \cite{R2} and obtain another characterization for S-Noetherian rings (see Theorem \ref{E:ann}). 
	Finally, in Theorem~\ref{E:nonzerodiv} we prove that a right $R$-module that is $S$-noetherian with respect to a set $S$ that contains no zero-divisors for $M$ contains a completely prime right annihilator; in case $R$ is commutative, this means that $M$ has an associated prime.

	   \section{Main Results} \label{E:Main}
	   
	     	Let $R$ be a ring. It is straightforward to see that finitely generated right ideals of $R$ are $S$-finite for any multiplicative subset $S$ of $R$. Therefore a right Noetherian ring is right $S$-Noetherian, as well. However, the converse does not always hold as the following example shows:
	     	
	     	\begin{exm}
			Let $R$ be a commutative ring and $S$ a multiplicative subset of $R$ such that $R$ is $S$-Noetherian but not Noetherian. (For example, one may take $R$ to be a non-Noetherian integral domain with $S = R \setminus \{0\}$; see~\cite[Proposition~2(a)]{A}.) Consider the ring $T$ of $2\times 2$ upper triangular matrices over $R$. Then $T$ is not right Noetherian since $R$ is not Noetherian (see \cite[Proposition 1.8]{GW}).The set $$S'=\left\{\begin{pmatrix}
			s & 0 \\ 
			0 & s
			\end{pmatrix}: s\in S\right\}$$ is a multiplicative subset of $T$. Let $J$ be a right ideal of $T$. Then, by \cite[Proposition 1.17]{L}, $J$ can be written as 
			$$\begin{pmatrix}
			I_1 & I_2 \\ 
			0 & I_3
			\end{pmatrix} $$
			where $I_1$, $I_2$, and $I_3$ are ideals of $R$ satisfying $I_1\subseteq I_2$. Since $R$ is $S$-Noetherian, there exist $s_i \in S$ and finitely generated ideals $F_i \subseteq I_i$ such that each $I_i s_i \subseteq F_i$ for $i = 1,2,3$. Without loss of generality, we may assume that $F_1 \subseteq F_2$. Then for $s = s_1 s_2 s_3 \in S$ we have
			$$J\begin{pmatrix}
			s & 0 \\ 
			0 & s
			\end{pmatrix}\subseteq \begin{pmatrix}
			F_1 & F_2 \\ 
			0 & F_3
			\end{pmatrix} = F,$$
			where $F$ is a finitely generated right ideal of $T$. Thus $T$ is right $S'$-Noetherian.\\
		
	     	\end{exm}
	    
	    We aim to prove that to test whether a ring is right $S$-Noetherian or not it is enough to control only completely prime right ideals. To this end, we first show that the set of all $S$-finite right ideals of a ring $R$ is a right Oka family.
	    
	    	\begin{lemma}\label{E:Oka}
	    		Let $R$ be a ring and $S\subseteq R$ a multiplicative subset. The set of all $S$-finite right ideals of a ring $R$ is a right Oka family.
	    	\end{lemma}
	    	\begin{proof}
	    		Since $R$ is finitely generated it is $S$-finite. So, we have $R\in \mathcal{F}$. Let $I$ be a right ideal of $R$ and $a\in R$ such that $I+aR$ and $a^{-1}I$ are $S$-finite. Then, there exist $s,t\in S$ and finitely generated right ideals $J,K$ of $R$ such that $(I+aR)s\subseteq J\subseteq I+aR$ and $(a^{-1}I)t\subseteq K\subseteq a^{-1}I$. 
Because $J \subseteq I+aR$, we may write its finitely many generators in the form $x_i + ar_i$ for some $x_i \in I, r_i \in R$. Then for the finitely generated right ideal $I_0 = \sum x_i R \subseteq I$, we have $J = \sum (x_i + ar_i) R \subseteq I_0 + aR$.
One may then verify from $Is \subseteq J \subseteq I_0 + aR$ that in fact $Is\subseteq I_0+a(a^{-1}I)$. Therefore we obtain
	    		$$Ist\subseteq I_0t+a(a^{-1}I)t\subseteq I_0+aK.$$
	    		Note that $I_0$ and $aK$ are a finitely generated right ideals contained in $I$. So from
	    		$$Ist\subseteq  I_0+aK\subseteq I,$$
	    		we conclude that $I$ is $S$-finite.
	    	\end{proof}
	    	For a family of right ideals $\mathcal{F}$ of a ring $R$, denote the complement of $\mathcal{F}$ within the set of the right ideals of $R$ as $\mathcal{F}'$.
	    	Let $\mathcal{F}$ be a right Oka family in a ring $R$. It is shown in \cite[Theorem 3.6] {R1} that if every nonempty chain of right ideals in $\mathcal{F}'$ has an upper bound in $\mathcal{F}'$ and every completely prime right ideal is in $\mathcal{F}$, then $\mathcal{F}$ contains all right ideals of $R$.
	    	\begin{theorem}\label{E:Cohen}
	    		Let $R$ be a ring and $S\subseteq R$ a multiplicative subset. Then $R$ is right $S$-Noetherian if and only if every completely prime right ideal of $R$ is $S$-finite. 
	    	\end{theorem}
	    	\begin{proof}
	    		If $R$ is right $S$-noetherian, then it is obvious that every completely prime right ideal of $R$ is $S$-finite. To prove the converse, let $\mathcal{F}$ be the set of all $S$-finite right ideals of $R$, and assume that every completely prime right ideal belongs to $\F$. By Lemma \ref{E:Oka}, the set $\mathcal{F}$ is a right Oka family. Let $\{I_\alpha\}_{\alpha\in A}$ be a nonempty chain in $\mathcal{F}'$. Set $I=\bigcup_{\alpha\in A}I_\alpha$. Clearly, $I$ is a right ideal. Assume that $I$ is $S$-finite. Then $Is\subseteq J\subseteq I$ for some $s\in S$ and finitely generated right ideal $J$ of $R$. Let $J=(x_1,..,x_k)$ for some $x_i\in R$. Since $J\subseteq I$ there is an index $\beta\in A$ such that $J\subseteq I_\beta$. Then, we get $I_\beta s\subseteq J\subseteq I_\beta$, which contradicts the assumption that $I_\beta$ is non-$S$-finite. Hence $I$ is an element of $\mathcal{F}'$. This shows that every nonempty chain in $\mathcal{F}'$ contains an upper bound in $\mathcal{F}'$. Because all completely prime right ideals of $R$ belong to $\F$, it follows from \cite[Theorem~3.6]{R1} that all right ideals of $R$ belong to $\mathcal{F}.$ Therefore $R$ is right $S$-Noetherian.
	    	\end{proof}

	    	Observe that a ring $R$ is right $S$-Noetherian if and only if the right $R$-module $R$ is $S$-Noetherian. For any short exact sequence
	    	$$0 \rightarrow M'\rightarrow M\rightarrow M''\rightarrow 0$$
	    	of right $R$-modules, $M$ is $S$-Noetherian if and only if $M'$ and $M''$ are $S$-Noetherian. In particular, if $R$ is right $S$-Noetherian, so is every finitely generated right $R$-module.
	    	
	    	Let $R$ be a ring with a multiplicative subset $S$, and let $M$ be a right $R$-module. We recall some notions introduced in \cite{AS2}, which extend routinely to the noncommutative setting. A chain of submodules $\{N_i\}_{i \in I}$ of a right $R$-module $M$ is said to be \emph{$S$-stationary} if there exists $j \in I$ and $s\in S$ such that $N_i s\subseteq N_j$ for all $i \in I$. Let $\mathcal{F}$ be a family of submodules of $M$. An element $N\in \mathcal{F}$ is called \emph{$S$-maximal} if there exists an $s\in S$ such that for each $L\in \mathcal{F}$, if $N\subseteq L$ then $Ls \subseteq N$. 

For a commutative ring $R$, Ahmed and Sana proved in \cite[Theorem 2.1]{AS2} that if $M$ is $S$-Noetherian, then every increasing sequence of extended submodules of $M$ is $S$-stationary. A submodule $N$ of $M$ is called \emph{extended} if there is an ideal $I$ of $R$ such that $N=IM$. In \cite[Corollary 2.1]{AS2}, the authors also obtained equivalent conditions for being an $S$-Noetherian ring when $S$ is finite: $R$ is $S$-Noetherian if and only if every increasing sequence of ideals of $R$ is $S$-stationary, if and only if every nonempty set of ideals of $R$ has an $S$-maximal element.  

We extend these observations to modules over noncommutative rings with no restrictions on the multiplicative subset, obtaining equivalent conditions for a module to be $S$-Noetherian. We define a family $\F$ of submodules of a right $R$-module $M$ to be \emph{$S$-saturated} if it satisfies the following property: for every submodule $N$ of $M$, if there exist $s \in S$ and $N_0 \in \F$ such that $Ns \subseteq N_0$, then $N \in \F$.
	    	
	    	\begin{theorem}\label{E:Noether}
	    		Let $S$ be a multiplicative subset of $R$ and $M$ a right $R$-module. The following are equivalent:
	    		\begin{enumerate}[(i)]
					\item $M$ is $S$-Noetherian.
					\item Every nonempty chain of submodules of $M$ is $S$-stationary.
					\item Every nonempty $S$-saturated set of submodules of $M$ has a maximal element.
					\item Every nonempty set of submodules of $M$ has an $S$-maximal element.
	    		\end{enumerate}  
	    	\end{theorem}
	    	
	    	\begin{proof}
				(i)$\Rightarrow$(ii) Let $M$ be an $S$-Noetherian right $R$-module and $\{K_i\}_{i \in I}$ a nonempty chain of submodules, indexed by $I$. The set $K=\bigcup_{i\in I}K_i$ is a submodule of $M$ and is $S$-finite. So, there is an $s\in S$ and a finitely generated submodule $F$ of $K$ such that $Ks\subseteq F\subseteq K$. Since $F$ is finitely generated, there is a $j\in I$ satisfying $F\subseteq K_j$. Then we have $Ks\subseteq F\subseteq K_j$, from which it follows that $K_i s\subseteq K_j$ for each $i\in I$.

				(ii)$\Rightarrow$(iii) Let $\F$ be an $S$-saturated set of submodules of $M$. Given any chain $\{N_i\}_{i \in I} \subseteq \F$, we claim that $N = \bigcup N_i$ is in $\F$, which will imply that $N \in \F$ is an upper bound for the chain. Indeed, by~(ii) there exists $N_j$ and $s \in S$ such that $N_i s \subseteq N_j$ for all $i \in I$. Thus $N s = \left(\bigcup N_i \right)s \subseteq N_j$. Because $\F$ is $S$-saturated, it follows that $N \in \F$ as desired. Now Zorn's lemma implies that $\F$ has a maximal element.

				(iii)$\Rightarrow$(iv) Let $\F$ be a nonempty set of submodules of $M$. Consider the family $\F^S$ of all submodules $L \subseteq M$ such that there exist some $s \in S$ and $L_0 \in \F$ with $Ls \subseteq L_0$. It is straightforward to see that $\F^S$ is $S$-saturated. Thus~(iii) implies that $\F^S$ has a maximal element $V \in \F^S$. Fix $s \in S$ and $N \in \F$ such that $Vs \subseteq N$. Now we claim that $N$ is an $S$-maximal element of $\F$; specifically, given $L \in \F$ with $N \subseteq L$, we will show that $L s \subseteq N$. Note that $V + L$ satisfies
\[
	(V+L)s = Vs + L s \subseteq N + L = L,
\]
so that $V+L \in \F^S$. Maximality of $V$ implies $V = V + L$, so that $L \subseteq V$. But then $L s \subseteq Vs \subseteq N$ as desired.
				
				(iv)$\Rightarrow$ (i) Suppose that (iv) holds. Let $N$ be a submodule of $M$, which we will prove to be $S$-finite. Let $\F$ denote the family of finitely generated submodules of $N$, which is nonempty as $0 \in \F$. Then $\F$ has an $S$-maximal element $F \in \F$. Fix $s \in S$ such that whenever $F \subseteq L$ with $L \in \F$, it follows that $Ls \subseteq F$. We claim that in fact $Ns \subseteq F$. Indeed, fixing $x \in N$, the submodule $L = F + xR$ of $N$ is finitely generated and contains $F$. Therefore $Ls \subseteq F$; in particular, $xs \in F$. This verifies $Ns \subseteq F$, so that $N$ is $S$-finite. Thus $M$ is $S$-Noetherian.
	    	\end{proof}

	    	In the study of noncommutative extensions of Cohen's Theorem in \cite{R2}, it proved useful to introduce point annihlator sets for classes of modules. A \emph{point annihilator} of a right $R$-module $M$ is an annihilator of a nonzero element $m$ of $M$. Let $\mathcal{C}$ be a class of right $R$-modules. A set $\mathcal{S}$ of right ideals of $R$ is a \emph{point annihilator set} for $\mathcal{C}$ if every nonzero $M\in\mathcal{C}$ has a point annihilator that lies in $\mathcal{S}$. In particular, a point annihilator set for the class of all right $R$-modules is said to be a right point annihilator set for $R$, and a point annihilator set for the class of all Noetherian right $R$-modules is called a right Noetherian point annihilator set for $R$. Note that if $R$ is a right Noetherian ring, a right point annihilator set for $R$ is the same as a right Noetherian point annihilator set for $R$.
	    	
	    	In a similar manner, we define right $S$-Noetherian point annihilator sets for $R$ as follows:
	    	\begin{definition}
	    		A point annihilator set for the class of all $S$-Noetherian right $R$-modules is said to be a right $S$-Noetherian point annihilator set for $R$.
	    		\end{definition} 
	    		Note that if $R$ is right $S$-Noetherian, then every right $S$-Noetherian point annihilator set for $R$ is a right point annihilator set for $R$.

	    Right $S$-Noetherian point annihilator sets do exist. The following example illustrates this concept.
	    
	    \begin{exm} Let $R$ be a commutative ring and $S$ a multiplicative subset of $R$. Let $\mathcal{F}_0$ be the set of all proper ideals of $R$ which meet $S$. We claim that $\mathcal{F}=Spec(R)\cup \mathcal{F}_0$ is an $S$-Noetherian point annihilator set for $R$. Let $M$ be a nonzero $S$-Noetherian $R$-module. For $R'=R/Ann_R(M)$, we have $M$ a faithful $S$-Noetherian $R'$-module, and so, $R'$ is $S$-Noetherian. The set $\mathcal{G}$ of ideals that are annihilators of nonzero elements of the $R'$-module $M$ is nonempty. So, by Theorem \ref{E:Noether}, there is a nonzero element $x\in M$ such that $I=Ann_{R'}(x)$ is $S$-maximal in $\mathcal{G}$. Note that $I$ is proper because $x \neq 0$. Assume that $I$ is not prime. We wish to show that $I\cap S\neq\varnothing$. Let $ab\in I$ for some $a,b\in R'$. Assume that $a\not \in I$. Then $ax\neq 0$. Then $Ann_{R'}(ax)$ properly contains the ideal $I$. Since $I$ is $S$-maximal, there is an $s\in S$ satisfying $Ann_{R'}(ax)s\subseteq I$. Then we have $bs\in I$. If $b\not\in I$, using a similar argument,  we get $st\in I$. Hence $I\cap S\neq \varnothing$.  Observe that $I=Ann_R(x)/Ann_R(M)$, and hence either $Ann_R(x)$ is prime or $Ann_R(x)\cap S\neq \varnothing$. This proves that $\mathcal{F}$ is an $S$-Noetherian point annihilator set for $R$.
	    	
	    	\end{exm}

	    Right $S$-Noetherian point annihilator sets provide another noncommutative generalization of Cohen's Theorem.
	    
	    \begin{theorem}\label{E:ann}
	    	Let $S$ be a multiplicative subset of $R$ and $\mathcal{T}$ is a right $S$-Noetherian point annihilator set for $R$. The following are equivalent:
	    	\begin{enumerate}[(i)]
	    		\item $R$ is right $S$-Noetherian.
	    		\item Every right ideal in $\mathcal{T}$ is $S$-finite.
	    		\item Every nonzero $S$-Noetherian right $R$-module has an $S$-finite point annihilator.
	    		\end{enumerate}
	    		 
	\end{theorem}
	\begin{proof}
		The family of $S$-finite right ideals is an Oka family by Lemma \ref{E:Oka}. Observe that every nonempty chain in $\mathcal{F}'$ has an upper bound in $\mathcal{F'}$. Let $\text{Max}(\mathcal{F}')$ denote the set of maximal elements in $\mathcal{F}'$. Let $I\in \text{Max}(\mathcal{F}')$, that is, $I$ is a maximal non $S$-finite right ideal. Any nonzero submodule of the right $R$-module $R/I$ is the image of a right ideal properly containing $I$, which must be $S$-finite, thus $R/I$ is an $S$-Noetherian right $R$-module. That means the set $\{R/I: I\in\text{Max}(\mathcal{F}')\}$ consists of $S$-Noetherian right $R$-modules. It is obvious that (i)$\Rightarrow$ (ii) and (i) $\Rightarrow$ (iii). For (ii)$\Rightarrow$ (i), suppose that every right ideal in $\mathcal{T}$ is $S$-finite. Then $\mathcal{T}\subseteq\mathcal{F}$. By \cite[Theorem 4.3(3)]{R2}, we conclude that every right ideal of $R$ belongs to $\mathcal{F}$. That is, the ring $R$ is $S$-Noetherian.  For (iii)$\Rightarrow$ (i), assume that every nonzero $S$-Noetherian right $R$-module has an $S$-finite point annihilator. Then $\mathcal{F}$ is a point annihilator set for the set $\{R/I: I\in\text{Max}(\mathcal{F}')\}$. Hence $\mathcal{F}$ consists of all right ideals of $R$, by \cite[Theorem 4.1(3)]{R2}. Thus $R$ is $S$-Noetherian.
	\end{proof}
	
	In \cite[Proposition 3.10]{R2}, it is shown that the set of completely prime right ideals of a ring $R$ is a right Noetherian point annihilator set for $R$. However, it is not always a right $S$-Noetherian point annihilator set for any multiplicative subset $S$ of $R$. For an extreme example, consider $T=\mathbb{C}[x_1,x_2,...]$, the polynomial ring over $\mathbb{C}$ in infinitely many variables and let $I=(x_1,x_2^2,x_3^3,..)$. Let $S$ be a multiplicative subset of the ring $R=T/I$, containing 0. Then every $R$-module $M$ will be trivially $S$-Noetherian. The ring $R$ does not possess any associated primes \cite[Example 5.2.6]{CR}. We conclude that no point annihilator of the the $S$-Noetherian module $M = R$ is (completely) prime.
	
	It is an open question for which multiplicative subsets $S$ of a ring $R$ the set of completely prime right ideals is an $S$-Noetherian point annihilator set for $R$. While we do not have an answer to this question, we will provide a positive result on the existence of a completely prime point annihilators for certain modules in Theorem~\ref{E:nonzerodiv} below. 

We will require the following result, which was stated without proof in~\cite[p.~3014]{R1}. We include a proof for the sake of completeness.

\begin{lemma}\label{E:Oka ann}
Let $R$ be a ring and let $M$ be a right $R$-module. The family 
	$$\mathcal{F} = \{I : \text{ for all } m\in M, mI = 0 \text{ implies } m = 0\}$$ 
of right ideals of $R$ is a right Oka family.
\end{lemma}

\begin{proof}
Clearly $R \in F$. Suppose that $I$ is a right ideal such that there exists $a \in R$ with $I+aR, a^{-1}I \in \F$. To prove that $I \in \F$, suppose that $m \in M$ with $mI = 0$. Because $(ma) \cdot a^{-1}I \subseteq mI = 0$ and $a^{-1}I \in \F$, we have $ma = 0$. But then $m(I+aR) \subseteq mI + maR = 0$, and $I+aR \in \F$ implies $m = 0$. Thus $I \in \F$ as desired.
\end{proof}

For a right $R$-module $M$ and an element $s \in R$, we say that $s$ is a \emph{non-zero-divisor} for $M$ if $ms = 0$ implies $m = 0$ for all $m \in M$.

\begin{theorem}\label{E:nonzerodiv}
Let $S$ be a multiplicative subset of a ring $R$, and let $M$ be a nonzero $S$-Noetherian right $R$-module. If every element of $S$ is a non-zero-divisor for $M$, then $M$ has a point annihilator that is a completely prime right ideal.
\end{theorem}

\begin{proof}
Let $\F$ denote the right Oka family of Lemma~\ref{E:Oka ann} for the module $M$. Fix a point annihilator $J_0 = Ann(m)$ for some nonzero $m$ in $M$, and note that $J_0 \in \mathcal{F}'$. We will show that $J_0$ is contained in a maximal element of $\mathcal{F}'$, which will be a point annihilator and completely prime right ideal, by \cite[Theorem 3.4]{R1}. 

Given a chain $\{J_i\}$ of right ideals in $\mathcal{F}'$ that all contain $J_0$, it suffices by Zorn's Lemma to show that $J = \bigcup J_i$ is in $\mathcal{F}'$. Note that the submodule $mR \subseteq M$ is isomorphic to $R/J_0$, so that each $K_i = J_i/J_0$ is isomorphic to the submodule $mJ_i$ of $M$. By Theorem \ref{E:Noether}, there is an $s\in S$ and an ideal $J_n$ such that every $(mJ_i)s$ is contained in $mJ_n$. This means that all $K_is$ are contained in $K_n$. It follows that all $J_is$ are in $J_n$, making $Js$ a subset of $J_n$.

	 Because $J_n \in \mathcal{F}'$, there exists $m_n \in M \setminus \{0\}$ with $m_n J_n = 0$. Then
	 $$m_n J s \subseteq m_n J_n = 0.$$
	 Because $s$ is a non-zero-divisor for $M$, we have $m_n J = 0$, and $J$ is in $\mathcal{F}'$ as desired.
	 Therefore $J$ is a completely prime right ideal as stated above, which is also a point annihilator of $M$.
\end{proof}
	 
For a commutative ring $R$, from the theorem above we obtain that every $S$-Noetherian $R$-module $M$ such that $S$ contains no zero-divisors for $M$ has an associated prime. This generalizes the well-known result every Noetherian $R$-module has an associated prime, and suggests that the $S$-Noetherian property with respect to a set $S$ of non-zero-divisors is a particularly strong finiteness property for a module. 

\bigskip
	
	\textbf{Acknowledgements.} We thank the anonymous referee for several helpful comments and suggestions. We also thank Omid Khani Nasab for valuable discussions, which led in particular to an improvement of Theorem~2.3.


\begin{thebibliography}{10}
	    	\bibitem{CR} Agrawal, S., et al. (n.d.) The CRing Project. Retrieved from \url{http://www.math.harvard.edu/~amathew/CRing.pdf}
	    	\bibitem{AS1} Ahmed, H., and Sana, H. (2015). $S$-Noetherian rings of the forms $\mathscr{A} [X]$ and $\mathscr{A}[[X]]$. \emph{Comm. Algebra} 43(9): 3848-3856.
	    	\bibitem{AS2} Ahmed, H., and Sana, H. (2016). Modules Satisfying the $S$-Noetherian property and $S$-ACCR. \emph{Comm. Algebra} 44(5): 1941-1951.
	    	\bibitem {A} Anderson, D. D., and T. Dumitrescu. (2002). $S$-Noetherian rings. \emph{Comm. Algebra} 30(9): 4407-4416.
	    	\bibitem{C} Cohen, I. S. (1950). Commutative rings with restricted minimum condition. \emph{Duke Math. J.} 17(1): 27-42.
	    	\bibitem {GW} Goodearl, K. R., and  Warfield Jr, R. B.  \emph{An Introduction to Noncommutative Noetherian Rings} (Vol. 61). Cambridge University Press, 2004.
	    	\bibitem{L} Lam, T. Y. \emph{A First Course in Noncommutative Rings} (Vol. 131). Springer Science \& Business Media, 2013.
	    \bibitem{LO}	Lim, J. W., and Oh, D. Y. (2015). S-Noetherian properties of composite ring extensions. \emph{Comm. Algebra} 43(7): 2820-2829.
	    	\bibitem {R1} Reyes, M. L. (2010). A one-sided prime ideal principle for noncommutative rings. \emph{J. Algebra Appl.} 9(6): 877-919.
	    	\bibitem {R2} Reyes, M. L. (2012). Noncommutative generalizations of theorems of Cohen and Kaplansky. \emph{Algebr. Represent. Theory} 15(5): 933-975.
		\bibitem{R3} Reyes, M.L. (2016). A prime ideal principle for two-sided ideals. \emph{Comm. Alg.} 44(11): 4585--4608.
	    	\bibitem {Z} Zhongkui, L. (2007). On $S$-Noetherian rings. \emph{Arch. Math.(Brno)} 43: 55-60.
	    	\end{thebibliography}
	    \end{document}